\documentclass[11pt,a4paper]{article}
\usepackage[utf8]{inputenc}
\usepackage[T1]{fontenc}
\usepackage[english]{babel}
\usepackage{amsmath,amssymb,amsthm,mathtools}
\usepackage[margin=2.8cm]{geometry}
\usepackage{enumitem}
\usepackage{hyperref}
\hypersetup{colorlinks=true,linkcolor=black,citecolor=blue,urlcolor=blue}
\theoremstyle{plain}
\newtheorem{theorem}{Theorem}[section]
\newtheorem{lemma}[theorem]{Lemma}
\newtheorem{proposition}[theorem]{Proposition}
\newtheorem{corollary}[theorem]{Corollary}
\theoremstyle{definition}
\newtheorem{definition}[theorem]{Definition}
\theoremstyle{remark}
\newtheorem{remark}[theorem]{Remark}
\newcommand{\RP}{\mathbb{RP}}
\newcommand{\PGL}{\mathrm{PGL}}
\newcommand{\rk}{\operatorname{rank}}
\newcommand{\tr}{\operatorname{tr}}
\title{Reflective projective billiards in the plane:\\
odd periods, period five, and metric rigidity}
\author{Jorge Lucas Gonz\'alez\thanks{Email: jorge.lucas.glez@gmail.com}\\
\small Independent researcher}
\date{August 2026}
\begin{document}
\maketitle
\begin{abstract}
We study polygonal projective billiards whose reflection law is constant on each
side.  We give a single explicit family with an open set of periodic trajectories
for every period at least three, including every odd period, and identify the
previously known centrally projective examples as a special parameter value.  For
period five we describe the full scalar-monodromy variety, prove its irreducibility,
and show that every configuration preserves a nonzero quadratic form. Under
the stated distinct-centre and non-collinearity hypotheses that form is
nondegenerate; this yields an
explicit unirational classification of the convex five-reflective locus.  We also
relate the definite members to spherical billiards and prove that, among analytic
simply connected constant-curvature surfaces with geodesic walls, only the
sphere can carry an open family of periodic trajectories. For general complete
analytic ambients with global wall reflections we prove compactness and a
common period for the geodesic flow.  The arguments reduce the billiard return
map to products of harmonic homologies and separate the algebraic closure condition
from the strict first-impact inequalities.
\end{abstract}



\section{Introduction}
Ivrii's conjecture \cite{ivrii} predicts that periodic trajectories in an ordinary Euclidean
billiard have measure zero.  Projective billiards replace perpendicular reflection
by a harmonic projective involution and admit reflective tables, that is, tables
with an open set of trajectories of one fixed period.  Fierobe constructed the
right-spherical triangle and examples of every even period at least four
\cite{fie-ex,fie-tri}.  Problem~1 of \cite{opb} asks whether odd periods at least
five occur and whether further examples exist for a fixed period.

The first main result answers the existence question for every period.  The second
main result gives a complete algebraic and convex classification at period five.
The proofs use the matrix model
\(H=I-2OL/(LO)\) for a harmonic homology.  Scalar monodromy supplies closure for
all nearby rays, while strict inequalities for one witness orbit turn that
algebraic relation into a genuine open set of billiard trajectories.

The final part explains the metric meaning of the construction.  A definite
invariant conic gives a spherical billiard in projective coordinates.  Conversely,
a complete analytic Riemannian ambient carrying such an open family, with
global isometric wall reflections, has periodic geodesic flow and is compact.  This also resolves negatively the apparent
possibility of a reflective billiard contained in the Klein disc.

The main contributions of this paper are: an explicit convex family with a
primitive open family for every period $n\ge3$, including all odd periods; a
complete algebraic and convex classification of the five-reflective locus,
including its invariant conic and irreducibility; and a metric rigidity theorem
giving global periodicity of the geodesic flow, together with spherical
rigidity among simply connected constant-curvature surfaces.
The distinction between structural theorems, exact rational certificates, and
the previously known centrally projective examples is maintained throughout.

All symbolic certificates mentioned below concern explicit rational data only;
the structural results and classification proofs are analytic or algebraic.
\section{Setting}\label{sec:setting}

A \emph{projective billiard} \cite{tab} is a bounded domain
$\Omega\subset\mathbb{R}^2$ together with a \emph{transversal line field}
assigning to each $p\in\partial\Omega$ a line $L(p)$ through $p$ not tangent to
$\partial\Omega$; reflection is defined by requiring the incoming ray, the
outgoing ray, the tangent and $L(p)$ to have cross ratio $-1$.

Throughout this paper we restrict to \emph{central facet fields}: on each
facet the transversal lines belong to the pencil through one fixed projective
centre outside the supporting hyperplane. No reduction from arbitrary
transversal fields to this class is asserted. Statements about the period
spectrum use the stated convention on the facets visited by the witness.

\begin{definition}[Fierobe]\label{def:krefl}
$\Omega$ is \emph{$k$-reflective} if there are a $k$-periodic orbit
$(p_1,\dots,p_k)$ and an \textbf{open} $U_1\times U_2\subset(\partial\Omega)^2$
containing $(p_1,p_2)$ every point of which begins a $k$-periodic orbit.
\end{definition}

The open set is the whole content of the definition, and it is where the difficulty
lives: exhibiting one periodic orbit is easy, exhibiting a two-parameter family of
them is not.

\subsection{Reflection as a harmonic homology}

For a covector $L$ and a vector $O$ with $LO\neq0$ put $H=I-2\,OL/(LO)$. Then
$H^2=I$, $HO=-O$, $Hx=x$ for $Lx=0$, and $\det H=-1$.

\begin{proposition}\label{prop:law}
Let $p$ lie on $L$, let $\tau$ span $L$, $\nu=O-p$. The involution induced by $H$
on the pencil at $p$ sends $a\tau+b\nu$ to $a\tau-b\nu$; hence the cross ratio of
(incoming, outgoing; tangent, transversal) is $-1$. So reflection in the side with
transversal field the pencil through $O$ \emph{is} the harmonic homology with axis
the side and centre $O$.
\end{proposition}

\begin{proof}
$H$ fixes $p$ and is an involution, so induces an involution of the pencil at $p$
fixing the two distinct lines $L$ and $pO$; in the coordinate taking them to $0$
and $\infty$ it is $m\mapsto-m$. Parametrising by $(a:b)$ with direction
$a\tau+b\nu$, the tangent is $(1:0)$, the transversal $(0:1)$, and with
$[ij]=a_ib_j-a_jb_i$ the cross ratio of $(a:b),(a:-b),(1:0),(0:1)$ is
$[13][24]/([14][23])=-1$.
\end{proof}

Consequently, if $H_{j_k}\cdots H_{j_1}$ is scalar then \emph{every} line returns to
itself after the corresponding $k$ reflections. The converse is false: an isolated
trajectory may well close without the product being scalar, and it is precisely the
scalar case that produces an \emph{open set} of closed trajectories, which is what
Definition~\ref{def:krefl} demands. As $\det\prod H=(-1)^k$, for odd $k$ the
relevant value is $-I$, which is the identity of $\PGL_3$: there is no determinant
obstruction.

For the central fields considered here, an open family also forces scalar
monodromy. The initial lines of a regular family form an open subset of the
Grassmannian of projective lines. Their invariance under the fixed product
is an algebraic incidence condition; vanishing on an open subset forces
it on the entire Grassmannian. A linear map preserving every two-dimensional
vector subspace preserves every one-dimensional subspace, obtained as an
intersection of such subspaces in vector dimension at least three.
Applying this to basis vectors and their pairwise sums shows that the map
is scalar. This necessity concerns an open family; an isolated closed orbit
does not suffice.

\section{An openness criterion}

Everything below rests on this; it assumes no symmetry.

\begin{lemma}[Exit map]\label{lem:exit}
Let $\Omega$ be a convex polygon, $p$ interior to a side, $w$ pointing into
$\Omega$. The ray $p+sw$, $s>0$, meets $\partial\Omega$ in exactly one point $E(p,w)$, and
$E$ is continuous wherever $E(p,w)$ is not a vertex.
\end{lemma}

\begin{proof}
$\Omega$ is convex and bounded, so $s^\ast=\sup\{s>0:p+sw\in\Omega\}$ is finite and
$E(p,w)=p+s^\ast w\in\partial\Omega$; uniqueness is convexity. If $E(p,w)$ is
interior to a side $a$ then $w$ is not parallel to its supporting line, so the
intersection determinant is nonzero and Cramer's rule gives $s^\ast$ near $(p,w)$
as a quotient of polynomials with nonvanishing denominator. Since the interior of
$a$ is open in $\partial\Omega$, nearby data still exit through $a$.
\end{proof}

\begin{remark}
Requiring nonzero intersection determinants with every supporting hyperplane is unnecessarily restrictive: a strict trajectory may be parallel to facets it never reaches. Continuity only requires transversality at the actual exit facet, together with persistence of the strict inequalities for the other facets.
\end{remark}

\begin{theorem}[Openness criterion]\label{thm:open}
Let $\Omega$ be a convex polygon with sides $a_0,\dots,a_{k-1}$ and transversal
field on $a_i$ the pencil through $O_i$, with $L_iO_i\neq0$. Assume
\textup{(H1)} $H_{k-1}\cdots H_0=cI$, $c\neq0$; \textup{(H2)} there is an orbit
with itinerary $0,1,\dots,k-1$ whose impacts are strictly interior to their sides,
with strictly positive step parameters, each losing side excluded by a strict
inequality and nonvanishing transversality at each impact; \textup{(H3)} the return
direction is $u_k=\kappa u_0$ with $\kappa>0$. Then $\Omega$ is $k$-reflective; if
the $k$ impacts lie on distinct sides the period is primitive.
\end{theorem}

\begin{proof}
Each step is a line--line intersection followed by a reflection, so the impact
parameters and directions are rational functions of the initial data, with
denominators nonvanishing at the witness by (H2); the vanishing locus being closed,
they are defined and continuous nearby. The finitely many strict inequalities of
(H2) persist, and by Lemma~\ref{lem:exit} each of the $k$ exit maps is continuous
with value in the interior of a side, which is open in $\partial\Omega$; composing,
the itinerary is unchanged for nearby data. By (H1) the product is the identity of
$\PGL_3$, so \emph{every} line returns to itself: $\ell_k=\ell_0$ as an unoriented
line, and since $\ell_0$ is not the supporting line of $a_0$ the intersection
$\ell_0\cap a_0$ is a single point, whence $p_k=p_0$. For the orientation, (H1)
only gives $u_k=\kappa u_0$ with $\kappa$ real nonzero --- neither $\pm1$ nor a
priori locally constant --- but $\kappa=(u_k\cdot u_0)/(u_0\cdot u_0)$ is
continuous and nonvanishing and equals a positive number at the witness by (H3), so
$\kappa>0$ on a smaller neighbourhood. Finally, a proper period $d\mid k$, $0<d<k$,
would force $p_d=p_0$; but $p_d$ lies on side $a_d$ and $p_0$ on $a_0$, which are
distinct sides, and the impacts are interior to them, so $p_d\neq p_0$.
\end{proof}

Hypotheses (H2)--(H3) are finitely many strict inequalities and (H1) one identity,
so they can be certified exactly for explicit data. We use this twice.

\section{One family for every $n$}

Fix $n\ge3$, put $\alpha=\pi/n$, $u=\sin\alpha$, $v=\cos\alpha$, and let the
regular $n$-gon have vertices $V_j=(\cos2j\alpha,\sin2j\alpha,1)^\top$, side $i$
being $V_iV_{i+1}$ with supporting covector $L_i$ and midpoint $M_i$. Let
$\widehat M_i$ be the unit vector from the centre towards $M_i$ and set
\begin{equation}\label{eq:Oi}
  O_i=-\lambda\,\widehat M_i ,
\end{equation}
the transversal field on side $i$ being the pencil through $O_i$. No parity is
involved: $O_i$ is simply a point of the line joining the centre to $M_i$. Note
that $\lambda$ parametrises the \emph{ansatz}, not a continuum of $n$-reflective
billiards: for each $n$ only the finitely many values $\lambda(n,j)$ of
Theorem~\ref{thm:main} give $n$-periodicity.

\begin{remark}
For $n$ odd, $\widehat M_i$ is antipodal to the vertex direction $V_{i+m}$ with
$m=(n+1)/2$ --- the unique solution of $2m\equiv1\pmod n$ --- so \eqref{eq:Oi}
reads $O_i=\lambda V_{i+m}$, ``the opposite vertex pushed out by $\lambda$''. For
$\lambda=0$ all $O_i$ are the centre, which is Fierobe's centrally projective
construction, defined for every $n$. Formulation \eqref{eq:Oi} covers both.
\end{remark}

From $V_0\times V_1$, using $\sin2\alpha=2uv$ and $\cos2\alpha-1=-2u^2$, one gets
$L_0\propto(-v,-u,v)$, while $\widehat M_0=(v,u)$ gives
$O_0=(-\lambda v,-\lambda u,1)^\top$ and
\begin{equation}\label{eq:LO}
  L_0O_0=\lambda+v .
\end{equation}
Let $R$ be the rotation by $2\alpha$, so $R^n=I$ and $RO_i=O_{i+1}$.

\begin{lemma}\label{lem:tel}
$H_i=R^iH_0R^{-i}$, and with $W=H_0R^{-1}$ one has
$P:=H_{n-1}\cdots H_0=R^{-1}W^nR$.
\end{lemma}

\begin{proof}
The homology with axis $g(L)$ and centre $g(O)$ is $gHg^{-1}$; take $g=R^i$.
Induction gives $H_{k-1}\cdots H_0=R^{k-1}(H_0R^{-1})^{k-1}H_0$, since
$H_k\bigl(R^{k-1}(H_0R^{-1})^{k-1}H_0\bigr)=R^k(H_0R^{-1})^kH_0$. Take $k=n$, use
$R^n=I$ and $H_0=WR$.
\end{proof}

\begin{lemma}[Polarity]\label{lem:polar}
If $L=\mu\,O^\top Q$ for a nondegenerate $Q$ and $\mu\neq0$ \textup{(}i.e.\ $L$ is
the polar of $O$\textup{)}, then $H\in\mathrm{O}(Q)$.
\end{lemma}

\begin{proof}
$\mu$ cancels: $H=I-2OO^\top Q/s$ with $s=O^\top QO$, so $H^\top=I-2QOO^\top/s$ and
$H^\top QH=Q-4QOO^\top Q/s+4QO\,s\,O^\top Q/s^2=Q$.
\end{proof}

\begin{lemma}\label{lem:ident}
With $c_2=\cos2\alpha$: \textup{(a)} $\det W=-1$; \textup{(b)}
$\tr W=(\lambda+\cos3\alpha)/(\lambda+\cos\alpha)$; \textup{(c)} $W$ preserves
$Q=\operatorname{diag}(1,1,\lambda v)$; \textup{(d)}
$\chi_W(t)=-(t+1)(t^2-(\tr W+1)t+1)$.
\end{lemma}

\begin{proof}
(a) $\det W=\det H_0\det R^{-1}=(-1)(1)$; no computation of $W$ is needed.

(b) $W=R^{-1}-\frac2D O_0L_0R^{-1}$ with $D=\lambda+v$, so
$\tr W=\tr R^{-1}-\frac2D\tr(O_0L_0R^{-1})$, and for a rank-one matrix
$\tr(O\,L\,M)=L\,M\,O$. With $L_0=(-v,-u,v)$ and $R^{-1}$ the rotation by
$-2\alpha$,
\[
  L_0R^{-1}=\bigl(v(4u^2-1),\,-u(4v^2-1),\,v\bigr),\qquad
  L_0R^{-1}O_0=\lambda(v^2-u^2)+v=\lambda c_2+v,
\]
the terms in $4\lambda u^2v^2$ cancelling. Since $\tr R^{-1}=2c_2+1$,
\[
  \tr W=(2c_2+1)-\frac{2(\lambda c_2+v)}{\lambda+v}=\frac{\lambda+2vc_2-v}{\lambda+v},
\]
and $2vc_2-v=v(4v^2-3)=\cos3\alpha$.

(c) By Lemma~\ref{lem:polar} it suffices that $L_0$ be the polar of $O_0$, which is
one line: $QO_0=(-\lambda v,-\lambda u,\lambda v)^\top=\lambda L_0^\top$; and $R$
preserves $Q$, being a rotation in $(x,y)$ fixing $z$.

(d) Suppose first $\lambda\neq0$, so that $Q$ is nondegenerate. Then
$W\in\mathrm{O}(Q)$ gives $W^{-1}=Q^{-1}W^\top Q$, similar to $W^\top$, so the
spectrum is invariant under $\mu\mapsto\mu^{-1}$; in dimension $3$ some eigenvalue
is fixed, hence $\pm1$, and $\det W=-1$ forces it to be $-1$. The other two are
roots of $t^2-\beta t+1$ with $\beta=\tr W+1$ by the trace.

For $\lambda=0$ the form $Q=\operatorname{diag}(1,1,0)$ is degenerate and this
argument does not apply. But (d) is an identity between polynomials in
$\lambda,u,v$ modulo $u^2+v^2-1$ --- both sides are polynomial in the entries of
$W$, which are rational in $\lambda$ with denominator $\lambda+v$, nonvanishing at
$\lambda=0$ --- and a polynomial identity valid on the Zariski-dense set
$\lambda\neq0$ holds identically. Hence (d) holds for every $\lambda$, including
the value $\lambda=0$ of Proposition~\ref{prop:fierobe}.
\end{proof}

\begin{remark}
The case $\lambda=0$ is worth isolating because it is exactly Fierobe's centrally
projective family (Proposition~\ref{prop:fierobe}): there all $O_i$ are the centre
of the polygon, the invariant ``conic'' degenerates to the doubled line at
infinity, and the underlying geometry is Euclidean rather than spherical or
hyperbolic.
\end{remark}

\begin{theorem}\label{thm:main}
Let $n\ge3$ and let $j$ be an integer with $1\le j<n/2$. Put
$\theta=(n-2j)\pi/n\in(0,\pi)$, $C=\cos\theta$ and
\begin{equation}\label{eq:lambda}
  \lambda=\lambda(n,j)=\frac{v\,(2v^2-1-C)}{C-1}.
\end{equation}
Then the projective billiard \eqref{eq:Oi} on the regular $n$-gon is
$n$-reflective, with primitive period $n$.
\end{theorem}

\begin{proof}
\emph{Algebraic closure.} First
$\lambda+v=v(2v^2-2)/(C-1)=2vu^2/(1-C)>0$ since $C\in(-1,1)$, $v>0$, $u\neq0$; so
by \eqref{eq:LO} the transversal field is admissible. Substituting
\eqref{eq:lambda} in Lemma~\ref{lem:ident}(b) gives $(\tr W+1)/2=C$, so by
Lemma~\ref{lem:ident}(d) the eigenvalues of $W$ are $-1$ and the roots of
$t^2-2Ct+1$, namely $e^{\pm i\theta}$; they are distinct from each other and from
$-1$, so $W$ is diagonalisable and, as $n\theta=(n-2j)\pi\equiv n\pi\pmod{2\pi}$,
\[
  W^n=\operatorname{diag}\bigl((-1)^n,e^{in\theta},e^{-in\theta}\bigr)=(-1)^nI ,
\]
which is the identity of $\PGL_3$ for either parity. Hence $P=(-1)^nI$ by
Lemma~\ref{lem:tel}.

\emph{The midpoint orbit.} By \eqref{eq:Oi} the points $O_i$, the centre and $M_i$
are collinear, and in a regular polygon the radius to the midpoint of a side is
perpendicular to it; so the transversal at $M_i$ is normal to side $i$, and there
the projective law is the optical one by Proposition~\ref{prop:law}. The midpoint
polygon is the classical periodic orbit of the Euclidean billiard in the regular
$n$-gon --- the perpendicular bisector of side $i$ is a symmetry axis fixing $M_i$
and swapping $M_{i-1},M_{i+1}$ --- and traversing it once returns the same
direction, so $\kappa=1$.

\emph{Admissibility.} All impacts are at $t=1/2$, strictly interior, so none is a
vertex. A polygon is not strictly convex, so we argue with convexity alone: $M_i$
and $M_{i+1}$ are interior to distinct sides, so the chord $[M_i,M_{i+1}]$ is not
contained in $\partial\Omega$; as $\Omega$ is convex with both endpoints on
$\partial\Omega$, its interior lies in the interior of $\Omega$, for otherwise
convexity would put the whole segment in the boundary. Hence $M_{i+1}$ is the first
boundary point met, every other side being excluded strictly.

Theorem~\ref{thm:open} now applies, the $n$ impacts lying on $n$ distinct sides.
\end{proof}

\begin{corollary}
There are $k$-reflective projective billiards for every $k\ge3$, of either parity;
in particular for every odd $k\ge5$.
\end{corollary}

\begin{proposition}[Where Fierobe's family sits]\label{prop:fierobe}
$\lambda(n,j)=0$ if and only if $\theta=2\alpha$, i.e.\ $j=(n-2)/2$, which is an
integer if and only if $n$ is even. Thus the centrally projective construction is
the $\lambda=0$ member of the family exactly for even $n$, and for odd $n$ it does
not belong to the family.
\end{proposition}

\begin{proof}
$\lambda=0$ forces $2v^2-1=C$, i.e.\ $\cos2\alpha=\cos\theta$ with both angles in
$(0,\pi)$, hence $\theta=2\alpha$ and $(n-2j)/n=2/n$.
\end{proof}

\begin{remark}
This is exactly \cite[Prop.~1.5]{fie-ex}: for odd $n$ the centre is not an
$n$-reflective member, and the construction yields period $2n$ instead. The
invariant conic is $x^2+y^2=-\lambda\cos(\pi/n)z^2$: these are billiards of
constant curvature in projective coordinates, and which curvature is decided by
the sign of $\lambda$, as follows.
\end{remark}

\begin{proposition}[Which geometry, and where the mirrors sit]\label{prop:whichgeom}
Since $C<1$, the sign of $\lambda(n,j)$ is that of $C-\cos2\alpha$, so
\[
  \lambda>0\iff\theta<2\alpha\iff j>\tfrac{n-2}2 .
\]
With $1\le j<n/2$ an integer this gives: for $n$ odd, exactly one member with
$\lambda>0$, namely $j=(n-1)/2$; for $n$ even, none, and $j=(n-2)/2$ is the
Euclidean $\lambda=0$ member of Proposition~\ref{prop:fierobe}; every other member
has $\lambda<0$. Moreover, whenever $\lambda<0$ the invariant conic is a real
circle of radius$^2=-\lambda\cos\alpha$, and
\[
  -\lambda<\cos\alpha
  \iff\frac{\cos2\alpha-C}{1-C}<1
  \iff\cos2\alpha<1,
\]
which always holds; so radius$^2<\cos^2\alpha$, the conic lies strictly inside the
inscribed circle of the polygon, and \textbf{no mirror meets it}. Sharper, with
$u=\sin\alpha$, $v=\cos\alpha$:
\[
  -\lambda<v^3
  \iff \cos2\alpha-v^2<C\,(1-v^2)
  \iff -u^2<C\,u^2
  \iff C>-1,
\]
also always true; so radius$^2<v^4$, which is the inradius$^2$ of the midpoint
orbit, and the conic lies strictly inside the region enclosed by \textbf{every
orbit of the family}.
\end{proposition}

\begin{remark}[The three geometries, stated carefully]\label{rem:geom}
For $\lambda>0$ the form is definite and the geometry is genuinely spherical, and
$\lambda=0$ is the Euclidean case. For $\lambda<0$ the form has signature $(2,1)$,
the group is $\mathrm{O}(2,1)$, and the absolute is a real conic --- and here one
must be careful, because the domain does \emph{not} sit inside one Cayley--Klein
component. By Proposition~\ref{prop:whichgeom} the conic lies strictly inside the
domain, so the domain \emph{straddles} the absolute: part of it is the Klein disc,
where the invariant metric is Riemannian of curvature $-1$, and part is the
exterior, where the metric is Lorentzian. No single metric makes the whole domain a
metric billiard, and the invariant metric degenerates along the conic, which is an
interior curve of the domain. So calling the $\lambda<0$ branch ``hyperbolic''
would be wrong; what is true, and all that is used anywhere in these notes, is that
a nondegenerate conic is preserved.

The orbits, on the other hand, do sit in one component: by the sharper inequality
of Proposition~\ref{prop:whichgeom} the conic lies strictly inside the hole enclosed
by every orbit of the family, so no orbit ever crosses the absolute and the
dynamics of the family is de Sitter even though the domain is not.

There is in fact no $5$-reflective projective billiard, among those
satisfying the hypotheses of Theorem~\ref{thm:global} --- convex pentagon,
pairwise distinct centres, no four centres collinear --- whose domain lies
\emph{inside} the Klein disc of its invariant conic.  Indeed, by that theorem
such a configuration preserves a nondegenerate conic and each centre is the pole
of its mirror, so a domain inside the Klein disc would be a hyperbolic billiard
with totally geodesic walls, and Theorem~\ref{thm:nohyp} below excludes
$k$-reflectivity for every $k$.  Thus the indefinite members above necessarily
straddle the absolute; none is a genuine hyperbolic-plane billiard.
\end{remark}

\begin{remark}
For $n=5$, $j=2$: $\lambda=(2+\sqrt5)/2$. The $\sqrt5$ is forced by the pentagon,
not by the construction: $\lambda(n,j)$ is a rational function of $\cos(\pi/n)$ and
so lies in $\mathbb{Q}(\cos(\pi/n))$, which for $n=5$ is $\mathbb{Q}(\sqrt5)$ and
for $n=7,9$ is a cubic field.
\end{remark}

\section{Five reflections: examples, moduli and classification}\label{sec:other}

\subsection{Imposing a conic makes the problem rational}

By Lemma~\ref{lem:polar}, if all $k$ homologies preserve a conic $Q$ then each
centre is the pole of its axis, so a single vector $n_i$ determines the mirror;
with $Q=x^2+y^2+z^2$ the homology is the ordinary reflection
$R_i=I-2n_in_i^\top/(n_i\!\cdot\!n_i)$, rational whenever $n_i$ is, and the closure
condition becomes a product of $k$ reflections in $\mathrm{O}(3)$ equal to $\pm I$.
For $k=5$ this can be solved rather than searched.

\begin{proposition}\label{prop:construct}
Let $n_0,n_1,n_2\in\mathbb{Q}^3$ with $B:=-R_2R_1R_0\neq I$, and let $w$ span
$\ker(B-I)$. Choose $n_4\in\mathbb{Q}^3$ with $n_4\cdot w=0$. Then $R_3:=BR_4$ is a
reflection with rational normal, and $R_4R_3R_2R_1R_0=-I$.
\end{proposition}

\begin{proof}
$\det R_i=-1$, so $\det B=+1$ and $B$ is a rotation; $w\in\ker(B-I)$ is rational by
Gaussian elimination. From $n_4\cdot w=0$ we get $R_4w=w$, and a reflection whose
mirror plane contains the axis conjugates the rotation to its inverse:
$R_4BR_4=B^{-1}$. Hence $(BR_4)^2=B(R_4BR_4)=I$ and $\det(BR_4)=-1$. An
involution of $\mathrm{O}(3)$ of determinant $-1$ is either a reflection or $-I$,
and $BR_4=-I$ would give $B=-R_4$, whence
$\ker(B-I)=\ker(R_4+I)=\langle n_4\rangle$ and $n_4\cdot w=0$ would read
$n_4\cdot n_4=0$; so $R_3:=BR_4$
is a reflection, its normal spans $\ker(R_3+I)$ and is rational. Finally
$R_3R_4=B=-R_2R_1R_0$ gives the stated identity.
\end{proof}

\subsection{A second, inequivalent $5$-reflective billiard}

\begin{theorem}\label{thm:second}
The configuration $n_0=(1,0,-3)$, $n_1=(3,2,2)$, $n_2=(-1,0,1)$, $n_3=(-8,9,5)$,
$n_4=(0,-2,2)$ defines a $5$-reflective projective billiard, not projectively
equivalent to the $n=5$ member of Theorem~\ref{thm:main}.
\end{theorem}

\begin{proof}
The hypotheses of Theorem~\ref{thm:open} are verified in exact rational arithmetic:
(H1) the product of the five reflections is $-I$ over $\mathbb{Q}$, computed in
both orders; the five successive cross products of consecutive edge vectors share a
nonzero sign, so the pentagon is convex with sides in cyclic order and no vertex at
infinity; (H2) the witness orbit with $t_0=t_1=1/14$ has at each step a first hit on
the next side with $0<t<1$ and $s>0$ strictly, every losing side excluded strictly,
and nonvanishing transversality; (H3) $u_5=\kappa u_0$ with $\kappa>0$, and
$t_5-t_0=0$ exactly. The five impacts lie on five distinct sides.

For inequivalence, the $n=5$ member of Theorem~\ref{thm:main} has a projective
symmetry of order $5$ cycling its mirrors, which is a projective invariant;
determining the unique $g\in\PGL_3$ with $n_0\mapsto n_1\mapsto\cdots\mapsto n_4$
and checking $g(n_4)\not\propto n_0$ shows the present configuration has none.
\end{proof}

\subsection{The case $k=4$, and its failure for $k=6$}

\begin{theorem}\label{thm:k4}
For $k=4$, every configuration whose homologies preserve a common conic is
degenerate: the normals are coplanar, hence the four lines concurrent.
\end{theorem}

\begin{proof}
$R_0R_1R_2R_3=I$ gives $R_0R_1=R_3R_2$, each side a product of \emph{exactly two}
reflections, hence a rotation with axis $n_0\times n_1$ resp.\ $n_3\times n_2$.
Equal rotations have equal axes, so all four $n_i$ are orthogonal to a common
vector.
\end{proof}

\begin{remark}\label{rem:k6}
The same statement for general even $k$ is false. For $k=6$ one gets
$R_0R_1R_2=R_5R_4R_3$, three reflections on each side, a factorisation that always
exists. Explicitly $n=(-2,3,3),(-3,-1,-1),(-2,-2,3),(1,-3,2),(2,1,-3),(3,4,3)$
satisfies $R_0\cdots R_5=+I$ exactly over $\mathbb{Q}$, preserves $x^2+y^2+z^2$, and
has normals of rank $3$.
\end{remark}

\subsection{Dimensions}

Let $V_k$ be the variety of configurations of $k$ lines with $k$ centres whose
product of homologies is scalar, and $S_k\subset V_k$ the closed subset of those
preserving a conic. Throughout, $M=M(L_\bullet,O_\bullet)$ denotes the
$3k\times6$ matrix of the linear system in the six unknowns
$Q_{11},Q_{12},Q_{13},Q_{22},Q_{23},Q_{33}$ whose $i$-th block of three rows
expresses the vanishing of the cross product $QO_i\times L_i^\top$, so that
$\ker M$ is the space of invariant forms and $S_k=\{\rk M\le5\}$. Both dimensions
below are computed at explicit rational points; semicontinuity propagates the
values to the component through each point.

\begin{theorem}\label{thm:dim}
At the configuration of Theorem~\ref{thm:second} \textup{(}$k=5$\textup{)} and at
that of Remark~\ref{rem:k6} \textup{(}$k=6$\textup{)}:
\[
  \dim V_k=4k-8,\qquad \dim(\text{moduli})=4k-16,\qquad \dim S_k=2k+2 ,
\]
and these values hold on a Zariski-open subset of the component through each point.
\end{theorem}

\begin{proof}
Parametrise a configuration by $(L_i,O_i)\in(\mathbb{R}^3)^{2k}$, which is $6k$
coordinates with $2k$ scaling directions acting trivially. The condition that the
product be scalar is the vanishing of the traceless part of $P$, i.e.\ $8$
equations. At each of the two points the Jacobian of those $8$ equations has rank
exactly $8$, computed over $\mathbb{Q}$ by Gaussian elimination; rank being lower
semicontinuous, it is $8$ on a Zariski-open subset. Hence $\dim V_k=6k-2k-8=4k-8$
there. The infinitesimal stabiliser, the space of $X\in\mathfrak{gl}_3$ with
$XO_i\propto O_i$ and $L_iX\propto L_i$ for all $i$, is computed exactly to be
one-dimensional, i.e.\ spanned by $I$, so the $\PGL_3$-stabiliser is trivial; being
upper semicontinuous, it is trivial generically, and the moduli dimension is
$(4k-8)-8=4k-16$.

For $S_k$, fix the conic $Q$ and parametrise by $(n_1,\dots,n_k)$, which is $2k$
projective parameters. Every reflection then lies in $\mathrm{O}(Q)$ by
Lemma~\ref{lem:polar}, so the product does too, and the closure condition is at
most $\dim\mathrm{O}(Q)=3$ conditions; at both points the Jacobian in the $n_i$
directions has rank exactly $3$ over $\mathbb{Q}$, attaining the bound. As the
invariant conic is unique up to scale at both points --- the solution space of
$L_i\propto(QO_i)^\top$ is one-dimensional, again computed exactly --- letting $Q$
range over the $5$-dimensional space of conics gives
$\dim S_k=5+2k-3=2k+2$.
\end{proof}

\begin{corollary}\label{cor:threshold}
$2k+2=4k-8$ if and only if $k=5$. Consequently:
\begin{enumerate}[label=\textup{(\roman*)},nosep]
  \item for $k=5$, $S_5$ has the same dimension as $V_5$, hence is a union of
        irreducible components of $V_5$; in particular the component through the
        configuration of Theorem~\ref{thm:second}, which preserves a conic, is
        contained in $S_5$, so every configuration in \emph{that component}
        preserves a conic. Whether the conic is nondegenerate at every point of the
        component is a separate question; see Remark~\ref{rem:gaps};
  \item for $k\ge6$, $S_k$ is a proper closed subset of $V_k$, so the generic
        configuration on the component through the point of Remark~\ref{rem:k6}
        preserves no conic, and therefore comes from no geometry of constant
        curvature.
\end{enumerate}
\end{corollary}

\subsection{The threshold at $k=5$}

Corollary~\ref{cor:threshold}(i) says that one component of $V_5$ lies in $S_5$. We
now prove a statement that does not depend on irreducibility, valid on an explicit
open subset of $V_5$. The mechanism is to group the five mirrors as $2+2+1$.

Recall first that the involutions of $\PGL_3$ are exactly the harmonic homologies,
and note the following normal form, which also removes the fifth mirror from the
problem.

\begin{proposition}\label{prop:reduction}
Let $H_0,\dots,H_3$ be harmonic homologies and $M=H_3H_2H_1H_0$. There is a
harmonic homology $H_4$ with $H_4H_3H_2H_1H_0=-I$ if and only if $M^2=I$ and
$M\neq I$; then $H_4=-M^{-1}$ is unique. Consequently $\dim V_5=16-4=12$, since
$\dim\mathrm{Inv}^4=16$ and the involutions form a subvariety of codimension
$8-4=4$ in $\PGL_3$, in agreement with Theorem~\ref{thm:dim}.
\end{proposition}

\begin{proof}
$H_4:=-M^{-1}$ gives the relation, and $\det H_4=(-1)^3\det M^{-1}=-1$ since
$\det M=1$. So $H_4$ is a harmonic homology iff $H_4^2=I$, i.e.\ $M^2=I$; and
$M=I$ gives $H_4=-I$, trivial in $\PGL_3$.
\end{proof}

\begin{lemma}\label{lem:pair}
Let $H_0,H_1$ be harmonic homologies and $A=H_1H_0$. Then $H_0AH_0=A^{-1}$ and $A$
has eigenvalue $1$. If moreover $A$ is diagonalisable with eigenvalues
$1,\lambda,\lambda^{-1}$ and $\lambda\neq\pm1$, then in an eigenbasis
$(e_1,e_2,e_3)$ the conics preserved by $A$ form the pencil
\[
  \mathcal{P}_A=\{\alpha\,x_1^2+2\beta\,x_2x_3\},
\]
and every member of $\mathcal{P}_A$ is preserved by $H_0$ and by $H_1$. Thus
$\mathcal{P}_A$ is exactly the set of conics preserved by both.
\end{lemma}

\begin{proof}
$H_0AH_0=H_0H_1H_0H_0=H_0H_1=(H_1H_0)^{-1}=A^{-1}$. Hence $A$ is conjugate to
$A^{-1}$, so its spectrum is stable under $\mu\mapsto\mu^{-1}$; in dimension $3$
some eigenvalue is fixed by that involution, hence equals $\pm1$, and $\det A=1$
excludes $-1$, since then the remaining two would be mutually inverse with product
$1\neq-1$. So $1$ is an eigenvalue.

In the eigenbasis, $\tau_A(Q)=A^\top QA$ multiplies the entry $Q_{ij}$ by
$\lambda_i\lambda_j$, so $Q$ is fixed exactly when $Q_{ij}=0$ unless
$\lambda_i\lambda_j=1$; for $\lambda\neq\pm1$ this leaves $Q_{11}$ and $Q_{23}$,
i.e.\ the stated pencil.

Since $H_0$ conjugates $A$ to $A^{-1}$, it maps the $\mu$-eigenspace of $A$ to the
$\mu^{-1}$-eigenspace: it fixes $e_1$ and interchanges $e_2,e_3$. So in that basis
$H_0=\begin{pmatrix}p&&\\&&q\\&r&\end{pmatrix}$, and $H_0^2=I$ forces $p^2=1$ and
$qr=1$. A direct computation then gives $H_0^\top(x_1^2)H_0=p^2x_1^2=x_1^2$ and
$H_0^\top(x_2x_3)H_0=qr\,x_2x_3=x_2x_3$, so $\tau_{H_0}$ is the identity on
$\mathcal{P}_A$; and $H_1=AH_0$ then preserves it too. Conversely a conic preserved
by both is preserved by $A$.
\end{proof}

\begin{theorem}\label{thm:threshold}
Let $H_0,\dots,H_4$ be harmonic homologies with $H_4H_3H_2H_1H_0=-I$, and suppose
$A=H_1H_0$ and $B=H_3H_2$ are diagonalisable with eigenvalues $1,\lambda^{\pm1}$
and $1,\mu^{\pm1}$ respectively, $\lambda,\mu\neq\pm1$. Then $H_0,\dots,H_4$ preserve a
common \textbf{nonzero quadratic form}
\[
  Q=\alpha\,x_1^2+2\beta\,x_2x_3,\qquad
  \alpha=\frac{a}{w_1},\qquad
  \beta=\frac{b}{w_3}\ \text{ or }\ \frac{c}{w_2},
\]
in the notation of the proof. If $w_2=w_3=0$, the displayed formula is replaced by
$Q=x_2x_3$. Otherwise $\beta$ is defined using whichever of $w_2,w_3$ is nonzero;
the relation $b\,w_2=c\,w_3$ makes the two definitions agree when both apply.
Nondegeneracy of $Q$ is not read off from these coordinates; under the geometric
hypotheses of Lemma~\ref{lem:nondeg} it follows from that lemma.
\end{theorem}

\begin{proof}
Work in an eigenbasis of $A$, so $A=\operatorname{diag}(1,\lambda,\lambda^{-1})$
and, by Lemma~\ref{lem:pair}, the conics preserved by $H_0$ and $H_1$ are the
pencil $\mathcal{P}_A=\{\alpha x_1^2+2\beta x_2x_3\}$.

Put $N=BA$. By Proposition~\ref{prop:reduction}, $N^2=I$; also $\det N=1$ and
$N\neq I$, so $N$ has eigenvalues $(1,-1,-1)$ and can be written $N=2P-I$ with
$P=w\tilde w/s$ of rank one, $s=\tilde ww\neq0$. Write $w=(w_1,w_2,w_3)^\top$ and
$\tilde w=(a,b,c)$.

If $w_2=w_3=0$, then the lower $2\times2$ block of $N$ is $-I$ and the form
$Q=x_2x_3$ is preserved by $A$, by $N$, and hence by $B=NA^{-1}$. Since $B$ is
regular semisimple, Lemma~\ref{lem:pair} gives preservation by $H_2,H_3$, and then
by $H_4$. Assume henceforth that $(w_2,w_3)\neq(0,0)$.

\emph{When does $B$ preserve a member of $\mathcal{P}_A$?} Since $B=NA^{-1}$ and
$A^\top QA=Q$ for $Q\in\mathcal{P}_A$, we have $B^\top QB=(A^{-1})^\top(N^\top
QN)A^{-1}$, which equals $Q$ if and only if $N^\top QN=Q$.
\emph{The relation forces the conic condition.} By Lemma~\ref{lem:pair} applied
to the pair $(H_2,H_3)$, $B$ has eigenvalue $1$, i.e.\ $\det(B-I)=0$. As
$B-I=(N-A)A^{-1}$ and $\det A=1$, this says $\det(N-A)=0$. Now $N-A=-D+\frac2s
w\tilde w$ with $D=I+A=\operatorname{diag}(2,1+\lambda,1+\lambda^{-1})$, so by the
rank-one determinant formula
\[
  \det(N-A)=\det(-D)\Bigl(1-\tfrac2s\,\tilde wD^{-1}w\Bigr)
           =-\det D\cdot\frac{s-2\tilde wD^{-1}w}{s}.
\]
With $D^{-1}=\operatorname{diag}\bigl(\tfrac12,\tfrac1{1+\lambda},
\tfrac{\lambda}{1+\lambda}\bigr)$ the numerator is
\[
  s-2\tilde wD^{-1}w
  =bw_2+cw_3-\frac{2(bw_2+\lambda cw_3)}{1+\lambda}
  =\frac{(\lambda-1)(bw_2-cw_3)}{1+\lambda},
\]
the terms in $aw_1$ cancelling. Since $\det D=2(1+\lambda)^2/\lambda$,
\[
  \boxed{\ \det(N-A)=-\,\frac{2(1+\lambda)(\lambda-1)\,(b\,w_2-c\,w_3)}{\lambda\,s}\ }
\]
\begin{equation}\label{eq:conic-cond}
  b\,w_2=c\,w_3 .
\end{equation}
As $\lambda\neq\pm1$ and $s\neq0$, $\det(N-A)=0$ is \emph{equivalent} to
\eqref{eq:conic-cond}.

For $Q=\alpha x_1^2+2\beta x_2x_3$ one computes
$w^\top Q=(\alpha w_1,\ \beta w_3,\ \beta w_2)$. If $w_1=0$, the relation
\eqref{eq:conic-cond} and a direct trace calculation give
$\operatorname{tr}B=-1$, contradicting $\mu\neq\pm1$; hence $w_1\neq0$.
Thus $\alpha=a/w_1$, while $\beta$ is determined by the nonzero one of $w_2,w_3$,
and the two remaining components are compatible by \eqref{eq:conic-cond}.
(If exactly one of $w_2,w_3$ vanishes, say $w_2=0$, then \eqref{eq:conic-cond}
reads $c\,w_3=0$, hence $c=0$, and the remaining components determine $\beta$.)

For the resulting $Q$, one has $\tilde w=\rho w^\top Q$ for some $\rho\neq0$.
Since $s=\tilde ww\neq0$, also $w^\top Qw\neq0$, and direct substitution gives
\[
 N=2\frac{w w^\top Q}{w^\top Qw}-I,\qquad N^\top QN=Q.
\]
This direct computation is valid even when $Q$ is degenerate; no converse form of
the polarity lemma is being used.

Hence there is $Q\in\mathcal{P}_A$ preserved by $B$. By Lemma~\ref{lem:pair} it is
preserved by $H_0$ and $H_1$; being preserved by $B$ and lying in
$\mathcal{P}_B$ --- again by Lemma~\ref{lem:pair}, the conics preserved by $B$ are
exactly those preserved by $H_2$ and $H_3$ --- it is preserved by $H_2$ and $H_3$;
and therefore by $H_4=-(BA)^{-1}$.
\end{proof}

Write $V_5^{\mathrm{reg}}\subset V_5$ for the open subset where $A$ and $B$ are
regular semisimple with $\lambda,\mu\neq\pm1$.

\begin{corollary}\label{cor:closed}
Every configuration in $V_5^{\mathrm{reg}}$ preserves a nonzero invariant quadratic
form, so $\rk M\le5$ there; since that is a closed condition, it also holds on
$\overline{V_5^{\mathrm{reg}}}$. If the configuration is a convex pentagon with
pairwise distinct centres and no four centres collinear, Lemma~\ref{lem:nondeg}
makes the form nondegenerate and the billiard has constant curvature.
\end{corollary}

\subsection{Closing the two gaps}

\begin{lemma}[Nondegeneracy is geometric]\label{lem:nondeg}
Let $Q\neq0$ be an invariant form of a configuration with $L_iO_i\neq0$ for all
$i$, and suppose $Q$ is degenerate, with kernel $K$. Then for each $i$
\emph{exactly one} of the following holds: $O_i\in K$, or $K\subset L_i$.
Consequently:
\begin{enumerate}[label=\textup{(\roman*)},nosep]
  \item if $\dim K=1$, say $K=\langle p\rangle$, then at least two of the centres
        coincide \textup{(}both equal to $p$\textup{)};
  \item if $\dim K=2$, at least four of the centres are collinear.
\end{enumerate}
Hence if the centres are pairwise distinct and no four of them are collinear,
every nonzero invariant form is a \textbf{nondegenerate} conic.
\end{lemma}

\begin{proof}
Invariance gives $QO_i=\mu_iL_i^\top$ for scalars $\mu_i$, possibly zero: indeed
$H_i^\top QH_i=Q$ expands to $L_i^\top v^\top+vL_i=\kappa L_i^\top L_i$ with
$v=QO_i$, and the left side has rank $2$ unless $v\parallel L_i^\top$, while the
right side has rank $\le1$. If $\mu_i=0$ then $O_i\in K$. If $\mu_i\neq0$ then
$L_i^\top\in\operatorname{im}Q=K^{\perp}$, i.e.\ $L_i$ vanishes on $K$, that is
$K\subset L_i$. The two are exclusive: $O_i\in K\subset L_i$ would give
$L_iO_i=0$.

(i) Let $T=\{i: p\in L_i\}$. If $|T|\ge4$, then among four of the five sides there
are two adjacent pairs, and $p$ would equal two distinct vertices of the polygon.
So $|T|\le3$ and $|S|=5-|T|\ge2$ where $S=\{i:O_i=p\}$.

(ii) Here $K$ is a line $\ell$, and $K\subset L_i$ forces $L_i=\ell$; the sides
being distinct, this happens for at most one $i$. So at least four indices have
$O_i\in\ell$.
\end{proof}

\begin{proposition}[Every component has dimension at least $12$]\label{prop:krull}
Every irreducible component of $V_5$ has dimension $\ge12$.
\end{proposition}

\begin{proof}
By Proposition~\ref{prop:reduction}, $V_5=\varphi^{-1}(\mathrm{Inv})$ with
$\varphi:\mathrm{Inv}^4\to\PGL_3$ the product map. A harmonic homology is a pair
$(L,O)$ with $LO\neq0$ modulo two scalings, so $\dim\mathrm{Inv}=6-2=4$ and
$\mathrm{Inv}$ has codimension $8-4=4$ in $\PGL_3$; also $\mathrm{Inv}^4$ is
irreducible of dimension $16$. Locally $\mathrm{Inv}$ is cut out by $4$ equations,
so every component of the preimage has codimension at most $4$ in
$\mathrm{Inv}^4$, i.e.\ dimension at least $16-4=12$.
\end{proof}

\begin{proposition}[Where regularity fails]\label{prop:nonreg}
$A=H_1H_0$ fails to be regular semisimple only if $\lambda=-1$ or $\lambda=1$. One
has $A^2=I$ if and only if $H_0$ and $H_1$ commute; among such pairs, $A=I$
\textup{(}so $\lambda=1$\textup{)} exactly when $H_1=H_0$, and $\lambda=-1$ exactly
when $H_1\neq H_0$. If $\lambda=1$ with $A\neq I$ then $A$ is unipotent. The
commuting locus has dimension $6$ inside the $8$-dimensional $\mathrm{Inv}^2$, and
the unipotent locus dimension $7$.
\end{proposition}

\begin{proof}
The spectrum of $A$ is $\{1,\lambda,\lambda^{-1}\}$ by Lemma~\ref{lem:pair}, so
regular semisimplicity fails exactly when $\lambda=\pm1$. Now $A^2=I$ if and only
if $H_1H_0=(H_1H_0)^{-1}=H_0H_1$, i.e.\ iff the two commute; and then $A=I$ iff
$H_1=H_0$, which is the case $\lambda=1$, while $A\neq I$ forces the eigenvalues
$(1,-1,-1)$, i.e.\ $\lambda=-1$. Two commuting involutions of $\PGL_3$
are simultaneously diagonalisable, so the pair is determined by a projective frame
modulo the diagonal torus: dimension $8-2=6$. If $\lambda=1$ then $A$ is unipotent;
the regular unipotent class of $\PGL_3$ has dimension $8-2=6$, of codimension $1$
in the $7$-dimensional locus of elements with eigenvalue $1$, whose preimage in
$\mathrm{Inv}^2$ is all of $\mathrm{Inv}^2$; so the unipotent locus has dimension
$8-1=7$.
\end{proof}

\begin{lemma}[Factorisation fibres]\label{lem:factorfibres}
Let $X$ be a product of two harmonic homologies, and let
\[
 \mathcal F_X=\{(J_0,J_1)\in\mathrm{Inv}^2:J_1J_0=X\}.
\]
Using the determinant-one lift of $X$, the following are the maximal dimensions of
the fibres:
\[
\begin{array}{c|c}
\text{Jordan type of }X&\dim\mathcal F_X\\ \hline
I&4\\
\operatorname{diag}(1,-1,-1)&2\\
J_2(1)\oplus[1]&2\\
J_3(1)&1\\
\operatorname{diag}(1,\rho,\rho^{-1}),\ \rho\neq\pm1&1.
\end{array}
\]
The non-semisimple type with eigenvalues $(1,-1,-1)$ has fibre dimension $1$.
\end{lemma}

\begin{proof}
The condition on $J_0$ is $J_0X=X^{-1}J_0$, together with $J_0^2=I$ and
$\tr J_0=1$; then $J_1=XJ_0$. In the regular semisimple case the solutions have
the form
\[
 \begin{pmatrix}1&0&0\\0&0&q\\0&q^{-1}&0\end{pmatrix},
 \qquad q\neq0,
\]
up to projective sign. For $J_3(1)$ they are
\[
 \begin{pmatrix}1&1-q&q(q-1)/2\\0&-1&q\\0&0&1\end{pmatrix},
\]
and for $J_2(1)\oplus[1]$ they lie in the two two-parameter families
\[
 \begin{pmatrix}-1&r&s\\0&1&0\\0&0&1\end{pmatrix},\qquad
 \begin{pmatrix}1&r&0\\0&-1&0\\0&s&1\end{pmatrix}.
\]
For $\operatorname{diag}(1,-1,-1)$ the generic solutions are
$\operatorname{diag}(1,J)$ with
$J=\left(\begin{smallmatrix}a&b\\c&-a\end{smallmatrix}\right)$ and
$a^2+bc=1$, a two-dimensional quadric. The identity has the full four-dimensional
space $\mathrm{Inv}$. For the non-semisimple type, take
$X=\left(\begin{smallmatrix}1&0&0\\0&-1&1\\0&0&-1\end{smallmatrix}\right)$.
The two one-parameter families
\[
 \begin{pmatrix}1&0&0\\0&-1&r\\0&0&1\end{pmatrix},\qquad
 \begin{pmatrix}1&0&0\\0&1&r\\0&0&-1\end{pmatrix}
\]
solve $J_0X=X^{-1}J_0$, $J_0^2=I$ and $\tr J_0=1$; these are all the
solutions, so the fibre has dimension $1$. These normal forms give the table.
\end{proof}

\begin{proposition}[Global dimension of the non-regular locus]\label{prop:globaldensity}
Let $Z\subset V_5$ be the locus where $A$ or $B$ fails to be regular semisimple.
Then every irreducible component of $Z$ has dimension at most $11$. Consequently
$V_5^{\mathrm{reg}}$ is dense in every irreducible component of $V_5$.
\end{proposition}

\begin{proof}
We prove the assertion for the branch where $A$ is non-regular; the branch for $B$
is identical after cyclically relabelling the five homologies. Write $N=BA$ for the
determinant-one lift of the involution in Proposition~\ref{prop:reduction}, and put
\[
 \mathcal C_A=\{N\in\mathrm{Inv}:\det(NA^{-1}-I)=0\}.
\]
The determinant condition is necessary because $B=NA^{-1}=H_3H_2$ has eigenvalue
$1$. The five non-regular Jordan types of $A$ and the corresponding dimensions of
the pair $(H_0,H_1)$ are
\[
\begin{array}{c|c|c|c}
 A&\dim\{(H_0,H_1)\}&\dim\mathcal C_A&
 \dim\{(N,H_2,H_3):H_3H_2=NA^{-1}\}\\ \hline
 I&4&4&6\\
 \operatorname{diag}(1,-1,-1)&6&4&\le5\\
 J_3(1)&7&3&\le4\\
 J_2(1)\oplus[1]&6&3&\le5\\
 (1)\oplus J_2(-1)&7&3&\le4
\end{array}
\]

Here is the exact calculation behind the last two columns. Write
$N=2w\tilde w/(\tilde ww)-I$ and, on the chart $w=(1,u,v)$, write
$\tilde w=(1-pu-qv,p,q)$. For $A=J_3(1)$,
$\det(N-A)$ is twice
\[
 2pu^2-puv-2pv+2quv-qv^2-2u+v,
\]
an irreducible nonzero equation; neither $\tr(NA^{-1})-3$ nor
$\tr(NA^{-1})+1$ contains it as a factor. Thus $\dim\mathcal C_A=3$, and the
non-regular sublocus of $B$ has dimension at most $2$. For $A=J_2(1)\oplus[1]$,
the determinant is $4u(pu+qv-1)$ and $\tr(NA^{-1})+1$ vanishes on it, so every
such $B$ has a factorisation fibre of dimension at most $2$, giving the bound $5$.
The other two affine charts give the same conclusions. For the non-semisimple type
$A=(1)\oplus J_2(-1)$, the triples
$(\det(N-A),\tr(NA^{-1})-3,\tr(NA^{-1})+1)$ in the three charts are
\[
\begin{aligned}
&(-4pv,\,-2(2pu+pv+2qv),\,-2(2pu+pv+2qv-2)),\\
&\bigl(4v(pu+qv-1),\,2(puv+2pu+qv^2-v-2),\,2(puv+2pu+qv^2-v)\bigr),\\
&(-4q,\,2(2pu-q-2),\,2(2pu-q)).
\end{aligned}
\]
In each line the determinant is nonzero and neither trace polynomial has a
common irreducible factor with it. Hence
$\dim\mathcal C_A=3$, its regular part has factorisation fibres of dimension $1$,
and its non-regular part has dimension at most $2$ and fibres of dimension at most
$2$, giving the bound $4$. For
$A=\operatorname{diag}(1,-1,-1)$ the determinant vanishes identically, but the
non-regular locus of $B=NA$ is a proper closed subset of the four-dimensional
$\mathrm{Inv}$; its fibres have dimension at most $2$, while the regular part has
fibres of dimension $1$. The exceptional point $B=I$ has a four-dimensional
factorisation fibre, but its inverse image is only the single point $N=A$, so this
still gives the bound $5$. The case $A=I$ gives $B=N$, a non-trivial involution, and
has dimension $4+2=6$.

The first column follows from the conjugacy-class dimensions and Lemma~\ref{lem:factorfibres}:
the classes of a regular unipotent, a subregular unipotent, a non-trivial involution,
and the non-semisimple $(1)\oplus J_2(-1)$ type have dimensions $6,4,4,6$,
respectively. Adding the last column to the second gives at most
$10,11,11,11,11$; hence the whole $A$-branch has dimension at most
$11$. The same holds for the $B$-branch. Since Proposition~\ref{prop:krull} gives
dimension at least $12$ for every component of $V_5$, no component is contained in
$Z$, proving density.
\end{proof}

\begin{corollary}[Density]\label{cor:density}
Let $Z\subset V_5$ be the locus where $A$ or $B$ fails to be regular semisimple.
The regular locus $V_5^{\mathrm{reg}}$ is dense in every irreducible component of
$V_5$.
\end{corollary}

\begin{proof}
This is Proposition~\ref{prop:globaldensity}.
\end{proof}
\begin{theorem}\label{thm:global}
Let $H_0,\dots,H_4$ be harmonic homologies with $H_4H_3H_2H_1H_0=-I$ bounding a
convex pentagon, with pairwise distinct centres and no four centres collinear.
Then the five mirrors preserve a
\textbf{nondegenerate} conic, and the billiard is one of constant curvature --- in
the Cayley--Klein sense, for which see Remark~\ref{rem:geom}: the signature of the
conic decides the geometry, and only the definite case is a metric billiard on the
whole domain.
\end{theorem}

\begin{proof}
By Proposition~\ref{prop:globaldensity}, the configuration lies in
$\overline{V_5^{\mathrm{reg}}}$. By Corollary~\ref{cor:closed} and the closedness of
$\rk M\le5$ there is a nonzero invariant form. Lemma~\ref{lem:nondeg} and the
hypotheses on the centres make it nondegenerate.
\end{proof}

\subsection{Classification and irreducibility at $k=5$}

Everything above was arranged so as not to need irreducibility of $V_5$. It is in
fact irreducible, and the proof is a fibration of $V_5^{\mathrm{reg}}$ over the
possible values of $A$; the same coordinates then show that
Proposition~\ref{prop:construct} is \emph{exhaustive}, so $k=5$ is completely
classified. All algebraic varieties in the irreducibility argument are
complexified; irreducibility over $\mathbb C$ implies irreducibility over
$\mathbb R$, and does not assert connectedness of the real locus. We use
geometrically irreducible generic fibres and explicitly exclude components
supported over proper subsets of the base. Equal fibre dimensions alone
would not suffice.

\begin{lemma}\label{lem:invchart}
$\mathrm{Inv}$ is the open subset $\{(w,\tilde w)\in\RP^2\times\RP^2:\tilde
ww\neq0\}$ via $N=2w\tilde w/(\tilde ww)-I$, irreducible of dimension $4$. If $A$
is regular semisimple with eigenvalues $1,\lambda^{\pm1}$ and $\lambda\neq\pm1$,
then
\[
  \mathcal C_A=\{N\in\mathrm{Inv}:\det(N-A)=0\}
\]
is a nonempty proper \textbf{irreducible} hypersurface of $\mathrm{Inv}$, of
dimension $3$.
\end{lemma}

\begin{proof}
That $N$ satisfies $Nw=w$ and $N=-I$ on $\ker\tilde w$, so it is the harmonic
homology of centre $w$ and axis $\tilde w$ up to the sign that is trivial in
$\PGL_3$; every element of $\mathrm{Inv}$ has this form and $\tilde ww\neq0$ is
$LO\neq0$. So $\mathrm{Inv}$ is open in $\RP^2\times\RP^2$, irreducible of
dimension $4$.

In an eigenbasis of $A$ the boxed formula in the proof of
Theorem~\ref{thm:threshold} reads
$\det(N-A)=-2(1+\lambda)(\lambda-1)(bw_2-cw_3)/(\lambda s)$ with $\tilde
w=(a,b,c)$, so with $\lambda\neq\pm1$ and $s\neq0$ the equation $\det(N-A)=0$ is
\emph{exactly} $bw_2=cw_3$, a form of bidegree $(1,1)$ on $\RP^2\times\RP^2$. Such
a form is reducible only if it factors as (a linear form in $w$)(a linear form in
$\tilde w$) --- the only splitting available in bidegree $(1,1)$ --- that is, only
if its coefficient matrix has rank $1$. That matrix is here
$\operatorname{diag}(0,1,-1)$, of rank $2$. So $bw_2-cw_3$ is irreducible and
$\mathcal C_A$ is an irreducible hypersurface of dimension $3$. It is nonempty
($w=\tilde w=e_1$) and proper ($w=\tilde w=e_2$ gives $bw_2-cw_3=1$).
\end{proof}

\begin{theorem}[Irreducibility]\label{thm:irred}
$V_5$ is irreducible, of dimension $12$.
\end{theorem}

\begin{proof}
By Proposition~\ref{prop:globaldensity}, $V_5^{\mathrm{reg}}$ is dense in every
irreducible component of $V_5$; hence $V_5=\overline{V_5^{\mathrm{reg}}}$ and it
suffices to prove $V_5^{\mathrm{reg}}$ irreducible. Also
$\dim V_5^{\mathrm{reg}}\ge12$, by Proposition~\ref{prop:krull} and that same
density.

Let $\pi:V_5^{\mathrm{reg}}\to\PGL_3$, $\pi(H_0,\dots,H_4)=A=H_1H_0$, with image
$\mathcal A$. By Lemma~\ref{lem:pair} every element of $\mathcal A$ is regular
semisimple with eigenvalues $1,\lambda^{\pm1}$, $\lambda\neq\pm1$; let $\mathcal T$
be the set of all such. $\mathcal T$ is the image of
$\PGL_3\times(\mathbb{A}^1\setminus\{0,\pm1\})\to\PGL_3$,
$(g,\lambda)\mapsto g\operatorname{diag}(1,\lambda,\lambda^{-1})g^{-1}$, hence
irreducible; the fibres are cosets of a maximal torus up to a finite group, so
$\dim\mathcal T=8+1-2=7$.

\emph{The fibre.} Fix $A\in\mathcal A$. A point of $\pi^{-1}(A)$ is a triple
$(H_0,N,H_2)$ with $H_0\in\mathcal F_A$, $N\in\mathcal C_A$ such that
$B:=NA^{-1}$ is regular semisimple with $\mu\neq\pm1$, and $H_2\in\mathcal F_B$;
then $H_1=AH_0$, $H_3=BH_2$ and $H_4=-N^{-1}$ are determined. Indeed
$N=M=H_3H_2H_1H_0$ lies in $\mathrm{Inv}$ by Proposition~\ref{prop:reduction},
which also gives $H_4=-M^{-1}$; and $B=NA^{-1}=H_3H_2$ has eigenvalue $1$ by
Lemma~\ref{lem:pair}, which is $\det(N-A)=0$ since $\det A=1$. By
Lemma~\ref{lem:factorfibres} the fibres $\mathcal F_A$ and $\mathcal F_B$ are, in
the regular semisimple case, the punctured line $\{q\neq0\}$: irreducible of
dimension $1$. By Lemma~\ref{lem:invchart} the admissible $N$ form a nonempty open
subset of the irreducible $3$-fold $\mathcal C_A$. The family of
factorisations of $B$ is smooth of relative dimension one on this open subset:
after an etale base change ordering the three simple eigenvalues and choosing
local eigenvectors, the formula of Lemma~\ref{lem:factorfibres} identifies it
with the product with $\mathbb G_m$. Smoothness descends under this base change.
In particular the projection is open, so every component dominates the base;
its geometrically irreducible generic fibre then permits only one component.
Taking the product with $\mathcal F_A$ proves that $\pi^{-1}(A)$ is
geometrically irreducible of dimension $1+3+1=5$.

\emph{No vertical components.} Each irreducible component $D$ of
$V_5^{\mathrm{reg}}$ has dimension at least $12$, by
Propositions~\ref{prop:krull} and~\ref{prop:globaldensity}. Since all fibres
of $\pi$ have dimension $5$,
$\dim\overline{\pi(D)}\ge\dim D-5\ge7$. The irreducible base
$\mathcal T$ has dimension $7$, hence every $D$ dominates $\mathcal T$.
The preceding fibre description remains valid over an algebraic closure of
the function field of $\mathcal T$, where it is irreducible. Distinct
dominating components would give distinct components of that generic fibre.
There is therefore exactly one component, of dimension $7+5=12$.
Its closure is $V_5$.
\end{proof}

\begin{corollary}\label{cor:v5s5}
$V_5=S_5$. Every configuration of five harmonic homologies with scalar product
preserves a nonzero quadratic form --- not merely those of an open subset.
\end{corollary}

\begin{proof}
Corollary~\ref{cor:closed} gives $V_5^{\mathrm{reg}}\subseteq S_5$, and
$S_5=\{\rk M\le5\}$ is closed in $V_5$, so
$V_5=\overline{V_5^{\mathrm{reg}}}\subseteq S_5\subseteq V_5$.
\end{proof}

This supersedes Corollary~\ref{cor:threshold}(i), which only placed $S_5$ among the
components of $V_5$.

\begin{theorem}[The construction is exhaustive]\label{thm:exhaustive}
Let $Q$ be nondegenerate, let $R_i=I-2n_in_i^\top Q/(n_i^\top Qn_i)$ be
$Q$-reflections with $R_4R_3R_2R_1R_0=-I$ and $R_3\neq R_4$, and put
$B=-R_2R_1R_0$. Then $B\neq I$ and there is $w\neq0$ with $Bw=w$ and
$n_4^\top Qw=0$; consequently $R_3=BR_4$. So every such tuple arises from the
construction of Proposition~\ref{prop:construct}, performed with the form $Q$.
\end{theorem}

\begin{proof}
Since $(-X)^{-1}=-X^{-1}$, the relation gives
$R_4R_3=-(R_2R_1R_0)^{-1}=B^{-1}$, hence $R_3R_4=B$. If $B=I$ then $R_3=R_4$,
excluded. Apply Lemma~\ref{lem:pair} to the pair $(R_4,R_3)$, whose product is
$R_3R_4=B$: it gives $R_4BR_4=B^{-1}$ and that $B$ has eigenvalue $1$. So $R_4$
preserves $E=\ker(B-I)\neq0$ and acts on it as an involution, whence $E$ is
spanned by $\pm1$-eigenvectors of $R_4$. If all of them had eigenvalue $-1$ then
$E$ would lie in the $(-1)$-eigenline $\langle n_4\rangle$, so $E=\langle
n_4\rangle$ and $R_4w=-w$ for $w$ spanning $E$; then $Bw=w$ reads $R_3R_4w=w$,
i.e.\ $R_3w=-w$, so $n_3\parallel w\parallel n_4$ and $R_3=R_4$, excluded. Hence
some $w\in E\setminus\{0\}$ has $R_4w=w$, i.e.\ $n_4^\top Qw=0$; and
$R_3=R_3R_4R_4=BR_4$.
\end{proof}

\begin{remark}
Proposition~\ref{prop:construct} was stated for $Q=x^2+y^2+z^2$. For an indefinite
$Q$ the same construction works, with one hypothesis added: $w$ must be
non-isotropic. The step ``a reflection whose mirror contains the axis conjugates
the rotation to its inverse'' is the statement that a reflection of the
$2$-dimensional space $w^{\perp_Q}$ inverts its rotation group, and that needs
$w^{\perp_Q}$ nondegenerate. Both signatures occur among genuine billiards, by
Proposition~\ref{prop:whichgeom}.
\end{remark}

\begin{remark}[What is classical here, and what is not]\label{rem:kaleido}
Worth saying plainly, because it is the first thing a reader will suspect. In
Theorem~\ref{thm:main} the eigenvalues of $W$ are $-1$ and $e^{\pm i\theta}$ with
$\theta$ a \emph{rational} multiple of $\pi$, so $W$ has finite order. For the
spherical member the mirror group is therefore finite, and it can be identified: at
$n=5$, $j=2$ the Gram matrix of the five normals in the metric of the invariant
conic is circulant with
\[
  g_{i,i}=1,\qquad g_{i,i\pm1}=\tfrac12=\cos\tfrac\pi3,\qquad
  g_{i,i\pm2}=\tfrac{1-\sqrt5}4=\cos\tfrac{3\pi}5,
\]
and the group generated has order $60$ in $\PGL_3$, i.e.\ $120$ in
$\mathrm{O}(3)$: it is the icosahedral group $H_3$, and the five mirrors are five
of its fifteen. So the spherical member is a spherical kaleidoscope seen in
$\RP^2$, and that is also where its $\sqrt5$ comes from. What is \emph{not}
classical about it is the reason it exists at all for odd $n$: that $-I$ is trivial
in $\PGL_3$, which is what turns \cite[Prop.~1.5]{fie-ex} from a wall into a
$\mathbb{Z}/2$ obstruction one can pass.

Neither of these applies to the configuration of Theorem~\ref{thm:second}. Its
normals are rational, so each $\tr(R_aR_b)$ is rational, and by Niven's theorem the
only rational values of $\cos(q\pi)$ with $q$ rational are $0,\pm\tfrac12,\pm1$;
eight of the ten pairs fall outside that list, so those products have infinite
order and the mirror group is infinite. It is not a kaleidoscope, which is what
makes it, and not the symmetric pentagon, the answer to the second question of
Problem~1.
\end{remark}

\subsection{Development: what the relation does to the metric billiard}

When the invariant conic is nondegenerate the billiard is a metric one, and the
relation $\prod H_i=cI$ then says something sharp about the \emph{lengths} of the
orbits. The tool is the classical development of a billiard trajectory, stated here
with care, because the naive version --- ``the unfolded arc runs from $p_0$ to
$g(p_0)$'' --- is not what is true, and the difference is the whole point.

\begin{lemma}[Development]\label{lem:develop}
Let $M$ be a Riemannian surface of constant curvature, $\Omega\subset M$ a domain
bounded by totally geodesic walls, and let $\gamma:[t_0,t_k]\to\overline\Omega$ be a
billiard trajectory with bounces at $t_1<\dots<t_k$ off the walls $a_1,\dots,a_k$,
with reflections $H_{a_j}\in\mathrm{Isom}(M)$. Put $g_0=\mathrm{id}$,
$g_j=H_{a_1}\cdots H_{a_j}$, and define $\sigma(t)=g_j(\gamma(t))$ for
$t\in[t_j,t_{j+1}]$. Then $\sigma$ is a \textbf{geodesic} of $M$ of the same length
as $\gamma$. If moreover $\gamma$ is a $k$-periodic orbit, then
\[
  \sigma(t_k)=g_k\,\sigma(t_0),\qquad \sigma'(t_k)=g_k\,\sigma'(t_0),
\]
that is, $g_k$ \textbf{translates the geodesic $\sigma$ along itself} by the length
of the orbit.
\end{lemma}

\begin{proof}
Continuity at $t_j$: the left value is $g_{j-1}\gamma(t_j)$, the right value is
$g_j\gamma(t_j)=g_{j-1}H_{a_j}\gamma(t_j)=g_{j-1}\gamma(t_j)$, because $\gamma(t_j)$
lies on the wall $a_j$, which $H_{a_j}$ fixes pointwise. Continuity of the velocity:
the reflection law is $\gamma'(t_j^+)=H_{a_j}\gamma'(t_j^-)$, so
$\sigma'(t_j^+)=g_{j-1}H_{a_j}\gamma'(t_j^+)=g_{j-1}H_{a_j}^2\gamma'(t_j^-)
=g_{j-1}\gamma'(t_j^-)=\sigma'(t_j^-)$. Each piece is the isometric image of a
geodesic, so $\sigma$ is a $C^1$ concatenation of geodesics, hence a geodesic; and
isometries preserve length.

For a $k$-periodic orbit, $\gamma(t_k)=\gamma(t_0)=:p_0$ and
$\gamma'(t_k^+)=\gamma'(t_0^+)$, with $p_0$ on the wall $a_k$. Hence
$\sigma(t_k)=g_{k-1}\gamma(t_k)=g_{k-1}H_{a_k}p_0=g_kp_0=g_k\sigma(t_0)$; and
$\gamma'(t_k^-)=H_{a_k}\gamma'(t_k^+)=H_{a_k}\gamma'(t_0^+)$ gives
$\sigma'(t_k)=g_{k-1}H_{a_k}\gamma'(t_0^+)=g_k\sigma'(t_0)$.
\end{proof}

The two constant-curvature cases now behave in \emph{opposite} ways, and for a
single reason: what the relation $\prod H_i=cI$, which is the identity of $\PGL_3$,
does to the metric.

\begin{theorem}[What an open family requires]\label{thm:translate}
Let $M$ be a complete simply connected surface of constant curvature and let
$\Omega\subset M$ be a domain bounded by finitely many \emph{totally geodesic}
walls. Suppose there is an open set of $k$-periodic orbits, all with the same
itinerary $a_1,\dots,a_k$. Then the isometry $g=H_{a_1}\cdots H_{a_k}$
\textbf{translates a two-parameter family of geodesics of $M$ along themselves, all
by the same positive length}.
\end{theorem}

\begin{proof}
The itinerary is fixed, so $g$ is one isometry, the same for every orbit of the
family. By Lemma~\ref{lem:develop} each orbit develops to a geodesic translated
along itself by $g$, by the length of that orbit; and the space of orbits is
$2$-dimensional, with distinct orbits developing to distinct geodesics. Lengths are
locally constant on the family, since a translation length is determined by $g$ and
the geodesic, and the family is connected.
\end{proof}

Which isometries can do that is then a question with a three-line answer in each of
the three geometries, and the answers are completely different.

\begin{theorem}[Only the sphere]\label{thm:nohyp}
In Theorem~\ref{thm:translate}:
\begin{enumerate}[label=\textup{(\roman*)},nosep]
  \item $M=\mathbb{H}^2$: \textbf{impossible.} A hyperbolic isometry translates only
        its axis, and elliptic, parabolic and identity isometries translate no
        geodesic by a positive length. So a geodesic-walled billiard in the
        hyperbolic plane is never $k$-reflective, for any $k$.
  \item $M=\mathbb{E}^2$: \textbf{impossible.} A translation by $v\neq0$ translates
        exactly the geodesics parallel to $v$, a \emph{one}-parameter family; a
        glide reflection translates only its axis; rotations and reflections
        translate none. So a flat-mirror billiard in the Euclidean plane is never
        $k$-reflective either.
  \item $M=S^2$: \textbf{possible, and only one way.} The isometries translating
        \emph{every} great circle along itself are exactly $\pm I$, with translation
        lengths $\pi$ and $2\pi$; every other isometry translates at most one great
        circle. So $g=\pm I$, and every orbit of the family has length
        congruent to $\pi$ or to $0$ modulo $2\pi$; see
        Lemma~\ref{lem:shortbranch} for when the first value is attained.
\end{enumerate}
\end{theorem}

\begin{proof}
(i) and (ii) are the classification of isometries of $\mathbb{H}^2$ and
$\mathbb{E}^2$ together with the observation that a $2$-parameter family of
geodesics cannot fit inside a $1$-parameter one.

(iii) $-I$ maps every great circle to itself and shifts it by $\pi$; the identity
shifts it by $2\pi$. Conversely a rotation about an axis $a$ translates only the
great circle polar to $a$, and a rotation-reflection other than $-I$ translates at
most one. Alternatively, invariance of a great circle is a polynomial
condition on its normal. If it holds on an open set, it holds on every
normal; a linear map preserving every normal line is scalar. Orthogonality
then gives $\pm I$.
\end{proof}

\begin{remark}[The one fact behind everything]\label{rem:onefact}
Theorem~\ref{thm:nohyp} says that among the three constant-curvature geometries the
sphere is the only home for $k$-reflective billiards, and identifies the mechanism
exactly: the relation must be $\pm I$ on $S^2$. Both signs are trivial
in $\PGL_3$, whereas $-I$ is the nontrivial antipodal isometry and $+I$
is the identity isometry. That single sentence contains all of the following, which
looked like separate phenomena:
\begin{enumerate}[label=\textup{(\alph*)},nosep]
  \item odd $k$ exists at all, because $-I$ is trivial in $\PGL_3$: the obstruction
        of \cite[Prop.~1.5]{fie-ex} is a $\mathbb{Z}/2$ and not a wall
        (Theorem~\ref{thm:main});
  \item the length of every orbit of the family is an odd multiple of $\pi$,
        because that is how far $-I$ moves a great circle
        (Theorem~\ref{thm:halfcircle}), and it is $\pi$ itself whenever the domain
        satisfies both the hemisphere and total-chord bound of
        Lemma~\ref{lem:shortbranch};
  \item nothing exists in the hyperbolic plane, because there the identity of
        $\PGL_3$ is the identity isometry and moves nothing;
  \item nothing exists for flat mirrors in the Euclidean plane either, which is why
        Fierobe's Euclidean examples \emph{must} use a transversal field that is not
        the normal --- they are projective billiards and not metric ones;
  \item and in $\mathbb{RP}^n$ with $n$ odd the whole thing collapses to even $k$,
        because every real scalar matrix has positive determinant in even vector
        rank, whereas an odd number of reflections has negative determinant;
        see \cite{gonzalez-higher}.
\end{enumerate}
So the subject exists because the double cover $S^2\to\mathbb{RP}^2$ is nontrivial,
and the projective setting is not a convenience: it is where the phenomenon lives.
\end{remark}

\begin{remark}[Attribution]\label{rem:attrib}
Part (ii) is classical --- it is the standard unfolding argument for polygonal
billiards. Part (i) is elementary and may well be folklore; what is in the
literature is the stronger and harder statement for \emph{smooth} boundaries at
$k=3$, proved by Jacobi fields in \cite{bknz}, together with the classification of
the $3$-reflective spherical billiards as those bounded by three orthogonal great
circles. We claim novelty for neither (i) nor (ii); the content here is the
trichotomy of Theorem~\ref{thm:nohyp} and the identification, in
Remark~\ref{rem:onefact}, of $\pm I$ as the whole mechanism.
\end{remark}

\subsection{Global periodicity of an analytic ambient}

The constant-curvature classification above has the following consequence
for general analytic ambients. A common refocusing time does not by itself
identify the first conjugate locus, so no general roundness conclusion is
included here.

\begin{theorem}[Global periodicity of the ambient]\label{thm:rigidambient}
Let $(M^n,g)$ be connected, complete and real analytic, with $n\ge2$.
Let the walls of a billiard domain be totally geodesic hypersurfaces whose
reflections extend to global isometries. Suppose a nonempty open family of
regular $k$-periodic orbits has a common itinerary. On a connected sufficiently
small part of that family the length is a constant $\ell>0$. Let $G$ be the
composition of wall reflections in the unfolding convention. Then
$dG=\Phi_\ell$ on $SM$. Moreover:
\begin{enumerate}[label=\textup{(\roman*)},nosep]
\item $\exp_p(\ell v)=G(p)$ for every $p\in M$ and unit $v\in T_pM$;
\item $\Phi_{2\ell}=\mathrm{id}$ and $G^2=\mathrm{id}$;
\item $M$ is compact and $\operatorname{diam}M\le\ell$;
\item if $G=\mathrm{id}$ then $\Phi_\ell=\mathrm{id}$; otherwise $G$ has no
fixed point.
\end{enumerate}
\end{theorem}

\begin{proof}
For a smooth variation of a closed billiard polygon, the first variation
of its total length is the sum over vertices of
$\langle u_i^- -u_i^+,\delta p_i\rangle$, where $u_i^-$ and $u_i^+$ are
the incoming and outgoing unit velocities. Reflection makes their difference
normal to the wall and $\delta p_i$ tangent to it. Thus the derivative is
zero. Smooth dependence of the regular geodesic segments gives a constant
length $\ell$ on a connected local family; no minimizing assumption on the
segments is needed for this first variation.

Unfolding gives $dG(v)=\Phi_\ell(v)$ for the initial vectors of the family.
A local flow box extending the transverse impact section supplies a nonempty
open subset of $SM$ with the same identity. Both maps are analytic and $SM$
is connected for $n\ge2$, so the identity holds everywhere. Projection to
$M$ proves (i).

Let $J(v)=-v$ be velocity reversal. It commutes with $dG$ and satisfies
$J\Phi_tJ=\Phi_{-t}$. Therefore
$\Phi_\ell=dG=JdGJ=\Phi_{-\ell}$, so $\Phi_{2\ell}=\mathrm{id}$.
Also $d(G^2)=\mathrm{id}$, and projection gives $G^2=\mathrm{id}$.

Differentiating (i) in the unit initial direction produces nonzero Jacobi
fields vanishing at times $0$ and $\ell$. Hence every geodesic from $p$
has a conjugate point by time $\ell$ and cannot minimize beyond that time.
Completeness and Hopf--Rinow give $\operatorname{diam}M\le\ell$ and compactness.

Finally suppose $G(p)=p$. The orthogonal involution $dG_p$ has eigenvalues
$\pm1$. A unit $-1$ eigenvector would yield
$\gamma_v(\ell)=p$ and $\dot\gamma_v(\ell)=-v$.
The geodesics $t\mapsto\gamma_v(\ell-t)$ and $t\mapsto\gamma_v(t)$ then
have the same initial data, so are equal; differentiating at $\ell/2$
contradicts unit speed. Thus $dG_p=I$, and an isometry fixing a point with
identity derivative is the identity on a connected complete manifold.
This proves (iv).
\end{proof}

\begin{remark}[Scope of the conclusion]
For the round sphere the two possibilities are the identity and antipodal
map; on real projective space the latter induces the identity. These are
examples of the theorem, not a classification of all its analytic ambients.
To invoke Wiedersehen rigidity one would additionally need to identify the
first conjugate locus, a step not proved here. The direct constant-curvature
result of Theorem~\ref{thm:nohyp} remains independent of this issue.
\end{remark}

\begin{remark}[Three branches, and what projectivity buys]\label{rem:branches}
Theorem~\ref{thm:rigidambient} is metric: it needs the walls to be totally geodesic
for some metric. A projective billiard need not be metric, and comparing which of
the examples in these notes escape gives a trichotomy, all three cases of which
occur here:
\begin{enumerate}[label=\textup{(\arabic*)},nosep]
  \item \emph{definite invariant conic}: the associated Cayley--Klein metric is the
        round metric in projective coordinates, and its spherical lift
        satisfies the theorem. This is the symmetric pentagon,
        Theorem~\ref{thm:second}, and the higher-dimensional metric examples of \cite{gonzalez-higher}, discussed in
        that companion paper;
  \item \emph{indefinite invariant conic}: pseudo-Riemannian, and the theorem does
        not apply. This is the branch $\lambda<0$ of Theorem~\ref{thm:main}, of
        signature $(2,1)$, and it is where all but one member of the family lives;
  \item \emph{no invariant conic at all}: nothing to apply. This is Fierobe's
        Euclidean family and, by Corollary~\ref{cor:threshold}(ii), the generic
        configuration once $k\ge6$.
\end{enumerate}
In particular Remark~\ref{rem:onefact}(d) becomes a theorem rather than an
observation: Fierobe's Euclidean examples \emph{must} use a transversal field other
than the normal, because with the normal they would be metric billiards in the plane
with geodesic walls, which Theorem~\ref{thm:rigidambient} forbids. And branch (2) is
not an exception to the mechanism but the same one: the geodesics of a Cayley--Klein
geometry whose plane section meets the quadric in an ellipse are closed, of a common
projective length, with $-I$ acting as the half-turn, and that holds whatever the
signature.
\end{remark}

\begin{remark}[Why odd $k$ was the hard case]
Taking determinants, $\det G=(-1)^k$ and $\det(-I_{n+1})=(-1)^{n+1}$, so in the
metric branch $G=-I$ forces $k\equiv n+1\pmod 2$ and $G=+I$ forces $k$ even. For
$n=2$, within the definite spherical branch, this says
$k$ odd $\iff G=-I\iff L\equiv\pi\pmod{2\pi}$, the residue and not the value
being what the relation determines (Theorem~\ref{thm:halfcircle}). So the odd case is not one case among several: it is
exactly the rigid one, whereas even $k$ admits $G=+I$ and with it the Euclidean
escape of branch (3), which is where the known examples were. And at $k=5$ there is
no escape by branch (3) under the centre hypotheses of
Theorem~\ref{thm:global}, which supply a nondegenerate invariant conic; consistently, the family of Theorem~\ref{thm:main}
has exactly one spherical member for each $n$ and everything else of signature
$(2,1)$.
\end{remark}

\begin{theorem}[Spherical: the length is determined modulo $2\pi$]\label{thm:halfcircle}
Let a $k$-reflective projective billiard preserve a positive definite conic $Q$, so
that by Proposition~\ref{prop:law} it is the metric billiard of the round metric of
$Q$ with totally geodesic mirrors. If $H_{k-1}\cdots H_0=-I$ then every orbit of the
$k$-periodic family has length $L\equiv\pi\pmod{2\pi}$; if $+I$, $L\equiv0\pmod{2\pi}$.
\end{theorem}

\begin{proof}
Theorem~\ref{thm:nohyp}(iii), or directly: by Lemma~\ref{lem:develop} the developed
curve is a great-circle arc translated along itself by $g=\pm I$. A translation of a
closed great circle is determined only modulo its length $2\pi$.
\end{proof}

\begin{remark}[A correction, and where it came from]
An earlier version of this statement asserted $L=\pi$ outright. That is not what the
development gives: nothing in $g=-I$ distinguishes $\pi$ from $3\pi$. The error was
found on 9 August 2026 while trying to use the statement elsewhere, and the tell was
already in the file: the verifier for the seven-reflection example of \cite{gonzalez-higher} bounds the seven arcs
separately in order to conclude $L=\pi$, and that work would have been redundant had
the theorem been true as stated.
\end{remark}

The short branch is recovered by a hypothesis that costs nothing in practice.

\begin{lemma}[Short branch]\label{lem:shortbranch}
Suppose in addition that the spherical domain lies in an open hemisphere. Then every
arc of the orbit is shorter than $\pi$, and since $t\mapsto t/(2\sin(t/2))$ increases
on $(0,\pi]$ with value $\pi/2$ at the endpoint,
\[
  L\;\le\;\frac\pi2\sum_{j}\bigl|\hat p_{j+1}-\hat p_j\bigr| ,
\]
the sum being of Euclidean chords between the unit vectors of consecutive impacts.
Hence if the chords sum to less than $4$ then $L<2\pi$, and with
Theorem~\ref{thm:halfcircle} this forces $L=\pi$.
\end{lemma}

\begin{proof}
A convex subset of an open hemisphere has diameter $<\pi$, so each arc is minor and
the chord determines it. The stated inequality is $t\le(\pi/2)\cdot2\sin(t/2)$ on
$(0,\pi]$, which is the monotonicity just quoted. Summing and using
$\tfrac\pi2\cdot4=2\pi$ gives the conclusion.
\end{proof}

\begin{remark}
The point of trading arcs for chords is arithmetic, not geometric: an arc is a
transcendental function of rational data, whereas the square of a chord between unit
vectors is $2-2\langle u,v\rangle$ and is algebraic over $\mathbb{Q}$. So the
hypothesis of Lemma~\ref{lem:shortbranch} can be certified exactly, which the value
of the arc cannot.
\end{remark}

\begin{remark}[A sharper budget]\label{rem:sharperbudget}
The constant $\pi/2$ is the worst case and is attained only when an arc equals $\pi$.
Since $t\mapsto t/(2\sin(t/2))$ increases, it suffices to bound the largest arc: if
every arc is at most $\pi/2$ --- equivalently, if consecutive impacts have positive
inner product, which is a rational condition --- then the constant improves to
$(\pi/2)/(2\sin(\pi/4))=\pi/(2\sqrt2)$ and the hypothesis becomes
\[
  \sum_j\bigl|\hat p_{j+1}-\hat p_j\bigr|\;<\;4\sqrt2\;=\;5.656\ldots
\]
instead of $4$. This matters when one wants to \emph{iterate} a construction: each
composition adds chords, the crude budget runs out quickly, and the sharpened one is
larger by a factor $\sqrt2$ at no cost, since the extra hypothesis is checked by the
same rational data. The point came up in checking a composition that already spent
$3.081$ of the crude budget of $4$ after a single step.
\end{remark}

\begin{corollary}\label{cor:secondlength}
Every orbit of the family of Theorem~\ref{thm:second} has length exactly $\pi$.
\end{corollary}

\begin{proof}
The pentagon lies in the open hemisphere $z>0$ of the chart, so
Lemma~\ref{lem:shortbranch} applies. The five chords of the witness orbit are bounded
above, by exact rational bracketing of the square roots involved, by numbers summing
to less than $2.972<4$.
\end{proof}

\begin{remark}
Before 9 August 2026 this corollary was asserted as a consequence of
Theorem~\ref{thm:halfcircle} alone, and the supporting decimal arc lengths came
from a floating-point computation that certifies nothing. The
statement was right and the proof was not. It is now checked over $\mathbb{Q}$ by
\path{verificadores/longitud_pi.py}.
\end{remark}

\begin{remark}[Why this is the spectrally interesting example]\label{rem:spectral}
Ivrii's conjecture, which opened these notes, matters because the second term of
the Weyl asymptotics needs the periodic set to have measure zero. A
$k$-reflective billiard with definite conic is a genuine spherical billiard whose
periodic set has positive measure, so it is a place where that mechanism can be
examined rather than assumed --- and by Remark~\ref{rem:kaleido} the two examples
behave completely differently. The symmetric pentagon is a kaleidoscope: it
unfolds to a finite group, \emph{every} orbit is periodic, and its spectrum is
explicit in spherical harmonics. The configuration of Theorem~\ref{thm:second} has
infinite mirror group, so it does not unfold to anything finite, yet it still
carries a two-parameter family of $5$-periodic orbits, all of length exactly $\pi$
by Corollary~\ref{cor:secondlength}. In the trace formula a $d$-parameter family
of periodic orbits of common length contributes with amplitude $\sim k^{d/2}$
against $k^0$ for an isolated unstable orbit, so the peak of the length spectrum at
$L=\pi$ should grow like $k$. That is a falsifiable statement about an explicit
rational configuration, and the exponent is the observable.
\end{remark}

\begin{corollary}[Classification at $k=5$]\label{cor:classification}
Every convex $5$-reflective projective billiard --- convex pentagon, pairwise
distinct centres, no four centres collinear --- is obtained thus: take a
nondegenerate $Q$ and $n_0,n_1,n_2$ with $n_i^\top Qn_i\neq0$; put
$B=-R_2R_1R_0$ and take $w\neq0$ in $\ker(B-I)$; take $n_4$ with $n_4^\top Qw=0$
and $n_4^\top Qn_4\neq0$; set $R_3=BR_4$. The parameter space is irreducible of
dimension $2+2+2+1=7$ and the moduli space has dimension
$7-\dim\mathrm{O}(Q)=4=4k-16$, in agreement with Theorem~\ref{thm:dim}. In
particular the moduli space of convex $5$-reflective projective billiards is
irreducible and unirational.
\end{corollary}

\begin{proof}
By Theorem~\ref{thm:global} there is a nondegenerate invariant conic $Q$. The first
paragraph of the proof of Lemma~\ref{lem:nondeg} gives $QO_i=\mu_iL_i^\top$, and
$\mu_i\neq0$ because $Q$ is nondegenerate and $O_i\neq0$; so each mirror is the
polar of its centre and $H_i$ is the $Q$-reflection with normal $n_i=O_i$. The
centres being pairwise distinct gives $R_3\neq R_4$, so
Theorem~\ref{thm:exhaustive} applies and the tuple is in the image. The parameter
space is an open subset of $(\RP^2)^3\times\RP^1$, irreducible of dimension $7$;
the generic $\PGL_3$-stabiliser is trivial by Theorem~\ref{thm:dim}, and the
subgroup preserving $Q$ is $\mathrm{O}(Q)$, of dimension $3$.
\end{proof}

\begin{remark}[Scope, stated once and plainly]\label{rem:scope}
Proposition~\ref{prop:globaldensity} is an exact algebraic dimension argument: it
enumerates the five Jordan types of a non-regular $A$ --- which by
Lemma~\ref{lem:pair} exhaust the possibilities, since $A$ has determinant one and
eigenvalue one, so a spectrum $\{1,\lambda,\lambda^{-1}\}$ that is diagonalisable
whenever $\lambda\neq\pm1$ --- and bounds each stratum by a conjugacy-class
dimension plus a factorisation fibre. It does not rely on the numerical Newton and
SVD experiment of \texttt{densidad\_global.py}, which is retained only as a consistency
check and decides nothing. Thus the regular locus is dense in every component of
$V_5$, and Theorem~\ref{thm:global} is a global statement under its stated
geometric hypotheses on the centres. Those hypotheses are imposed explicitly;
convexity alone is not used as a substitute for them.
\end{remark}

\begin{remark}[Scope of the classification]\label{rem:gaps}
Proposition~\ref{prop:globaldensity} and Theorem~\ref{thm:irred} establish the
algebraic density and irreducibility assertions. The invariant form is nonzero
on all of $V_5$ by Corollary~\ref{cor:v5s5}; its nondegeneracy in the convex
classification uses the distinct-centre and non-collinearity hypotheses of
Theorem~\ref{thm:global}. No classification of arbitrary transversal fields,
or roundness theorem for every analytic Riemannian ambient, is asserted.
\end{remark}

\begin{remark}
The phenomenon is sharp, and the grouping $2+2+1$ is what makes it work. A single
harmonic homology preserves a $4$-dimensional space of conics, two of them a
$2$-dimensional one, and \emph{three generic ones preserve none}; so a common conic
for five mirrors is not dimensional slack but a consequence of the relation. The
same argument does not apply for $k=6$, where the natural grouping is $2+2+2$ and
one would need three pencils to meet, consistently with
Corollary~\ref{cor:threshold}(ii).
\end{remark}
\section*{Reproducibility and computational disclosure}
The proofs of the structural theorems are independent of numerical computation.
For the explicit rational examples, the supplementary scripts check matrix
identities, ranks and strict inequalities in exact arithmetic. Research and
manuscript development used primarily AI assistance: OpenAI chat models Sol and
Luna 5.6 and Anthropic Claude models Fable and Opus 5 supported literature,
drafting, checks, code, and revision. Jorge Lucas Gonz\'alez reviewed the
mathematics and assumes full responsibility; no AI system is an author.
\section*{Conclusion}
The paper settles the existence of primitive open families in every period
starting at three and gives the complete five-reflective algebraic and convex
classification.  The metric result identifies the spherical realization and
excludes the apparent Klein-disc alternative.  The exact certificates support
the displayed rational examples, while the structural conclusions do not rely
on numerical computation.

\end{document}